\documentclass[12pt, reqno]{amsart}%
\usepackage{amsmath, amsthm, amscd, amsfonts, amssymb, graphicx, color}
\usepackage[bookmarksnumbered, colorlinks, plainpages]{hyperref}
\usepackage{amsmath}
\usepackage{amsfonts}
\usepackage{amssymb}
\usepackage{graphicx}%
\providecommand{\U}[1]{\protect \rule{.1in}{.1in}}
\newtheorem{theorem}{Theorem}[section]

\newtheorem{proposition}[theorem]{Proposition}
\newtheorem{corollary}[theorem]{Corollary}
\theoremstyle{definition}
\newtheorem{definition}[theorem]{Definition}

\theoremstyle{remark}
\newtheorem{remark}[theorem]{Remark}
\numberwithin{equation}{section}
\begin{document}
\title[Bertrand Curves in three Dimensional Lie Groups]{Bertrand Curves in three Dimensional Lie Groups}
\author{O. Zeki Okuyucu$^{^{(1)}}$}
\address{$^{^{(1)}}$Bilecik \c{S}eyh Edebali University, Faculty of Science and Arts,
Department of Mathematics, 11210, Bilecik, Turkey}
\email{osman.okuyucu@bilecik.edu.tr}
\author{\.{I}smail G\"{o}k$^{^{(2)}}$}
\address{$^{^{(2)}}$Ankara University, Faculty of Science, Department of Mathematics,
06100, Tando\~{g}an, Ankara, Turkey}
\email{igok@science.ankara.edu.tr}
\author{Yusuf Yayl\i$^{^{(3)}}$}
\address{$^{^{(3)}}$Ankara University, Faculty of Science, Department of Mathematics,
06100, Tando\~{g}an, Ankara, Turkey}
\email{yayli@science.ankara.edu.tr}
\author{Nejat Ekmekci$^{^{(4)}}$}
\address{$^{^{(4)}}$Ankara University, Faculty of Science, Department of Mathematics,
06100, Tando\~{g}an, Ankara, Turkey}
\email{ekmekci@science.ankara.edu.tr}
\thanks{This paper is in final form and no version of it will be submitted for
publication elsewhere.}
\date{September 08, 2012}

\begin{abstract}
In this paper, we give the definition of \ harmonic curvature function some
special curves such as helix, slant curves, Mannheim curves and Bertrand
curves. Then, we recall the characterizations of helices \cite{ciftci}, slant
curves (see \cite{zeki}) and Mannheim curves (see \cite{ismail}) in three
dimensional Lie groups using their harmonic curvature function.

Moreover, we define Bertrand curves in a three dimensional Lie group $G$ with
a bi-invariant metric and the main result in this paper is given as (Theorem
\ref{teo 3.2}): A curve $\alpha:I\subset \mathbb{R\rightarrow}G$ with the
Frenet apparatus $\left \{  T,N,B,\varkappa,\tau \right \}  $ is a Bertrand curve
if and only if%
\[
\lambda \varkappa+\mu \varkappa H=1
\]
where $\lambda$, $\mu$ are constants and $H$ is the harmonic curvature
function of the curve $\alpha.$

\end{abstract}
\maketitle

\setcounter{page}{1}

\setcounter{page}{1}


\setcounter{page}{1}


\section{Introduction}

The general theory of curves in a Euclidean space (or more generally in a
Riemannian manifolds) have been developed a long time ago and we have a deep
knowledge of its local geometry as well as its global geometry. In the theory
of curves in Euclidean space, one of the important and interesting problem is
characterizations of a regular curve. In the solution of the problem, the
curvature functions $k_{1}\left(  \text{or }\varkappa \right)  $ and
$k_{2}\left(  \text{or }\tau \right)  $ of a regular curve have an effective
role. For example: if $k_{1}=0=$ $k_{2}$, then the curve is a geodesic or if
$k_{1}=$constant$\neq0$ and $k_{2}=0,$ then the curve is a circle with radius
$\left(  1/k_{1}\right)  $, etc. Thus we can determine the shape and size of a
regular curve by using its curvatures. Another way in the solution of the
problem is the relationship between the Frenet vectors of the curves (see
\cite{kuh}). For instance Bertrand curves:

In the classical diferential geometry of curves, J. Bertrand studied curves in
Euclidean 3-space whose principal normals are the principal normals of another
curve. In (see \cite{bert}) he showed that a necessary and sufficient
condition for the existence of such a second curve is that a linear
relationship with constant coefficients shall exist between the first and
second curvatures of the given original curve. In other word, if we denote
first and second curvatures of a given curve by $k_{1}$ and $k_{2}$
respectively, then for $\lambda,\mu$ $\in \mathbb{R}$ we have $\lambda
k_{1}+\mu k_{2}=1$. Since the time of Bertrand's paper, pairs of curves of
this kind have been called \textit{Conjugate Bertrand} \textit{Curves}, or
more commonly \textit{Bertrand Curves} (see \cite{kuh})$.$

In 1888, \textit{C. Bioche} \cite{bioche} give a new theorem to obtaining
Bertrand curves by using the given two curves $C_{1}$ and $C_{2}$ in Euclidean
$3-$space. Later, in 1960, \textit{J. F. Burke} \cite{burke} give a theorem
related with \textit{Bioche's thorem on Bertrand curves.}

The following properties of Bertrand curves are well known: If two curves have
the same principal normals, \textbf{(i)} corresponding points are a fixed
distance apart; \textbf{(ii)} the tangents at corresponding points are at a
fixed angle. These well known properties of Bertrand curves in Euclidean
3-space was extended by \textit{L. R. Pears} in \cite{Pears} to \ Riemannian
$n-$space and found general results for Bertrand curves. When we applying
these general result to Euclidean $n$-space, it is easily find that either
$k_{2}$ or $k_{3}$ is zero; in other words, Bertrand curves in ,$\mathbb{E}%
^{n}(n>3)$ are degenerate curves. This result is restated by \textit{Matsuda
and Yorozu }\cite{mat}. They proved that \textit{there is no special Bertrand
curves }in\textit{\ }$E^{n}(n>3)$ and they define new kind, which is called
$\left(  1,3\right)  -$type Bertrand curves in $4-$dimensional Euclidean
space. Bertrand curves and their characterizations were studied by many
authours in Euclidean space as well as in Riemann--Otsuki space, in Minkowski
3- space and Minkowski spacetime (for instance see \cite{bal, bal1, il, jin,
kul, james, yil}.)


The degenarete semi-Riemannian geometry of Lie group is studied by
\c{C}\"{o}ken and \c{C}ift\c{c}i \cite{coken}. Moreover, they obtanied a
naturally reductive homogeneous semi-Riemannian space using the Lie group.
Then \c{C}ift\c{c}i \cite{ciftci} defined general helices in three dimensional
Lie groups with a bi-invariant metric and obtained a generalization of
Lancret's theorem. Also he gave a relation between the geodesics of the
so-called cylinders and general helices. Then, Okuyucu et al. \cite{zeki}
defined slant helices in three dimensional Lie groups with a bi-invariant
metric and obtained some characterizations using their harmonic curvature function.


Recently, \textit{Izumiya and Takeuchi} \cite{izu} have introduced the concept
of slant helix in Euclidean $3$-space.\ A slant helix in Euclidean space
$\mathbb{E}^{3}$ was defined by the property that its principal normal vector
field makes a constant angle with a fixed direction. Also, Izumiya and
Takeuchi showed that $\alpha$ is a slant helix if and only if the geodesic
curvature of spherical image of principal normal indicatrix $\left(  N\right)
$ of a space curve $\alpha$
\[
\sigma_{N}\left(  s\right)  =\left(  \frac{\varkappa^{2}}{\left(
\varkappa^{2}+\tau^{2}\right)  ^{3/2}}\left(  \frac{\tau}{\varkappa}\right)
^{\prime}\right)  \left(  s\right)
\]
is a constant function .


Harmonic curvature functions were defined earlier by \"{O}zdamar and Hac\i
saliho\u{g}lu \cite{ozdamar}. Recently, many studies have been reported on
generalized helices and slant helices using the harmonic curvatures in
Euclidean spaces and Minkowski spaces \cite{camci, gok, kulahci}. Then,
Okuyucu et al. \cite{zeki} defined slant helices in three dimensional Lie
groups with a bi-invariant metric and obtained some characterizations using
their harmonic curvature function.


In this paper, first of all, we give the definition of \ harmonic curvature
function some special curves such as helix, slant curves. Then, we recall the
characterizations of helices \cite{ciftci}, slant curves (see \cite{zeki}) and
Mannheim curves (see \cite{ismail}) in three dimensional Lie groups using
their harmonic curvature function. Moreover, we define Bertrand curves in a
three dimensional Lie group $G$ with a bi-invariant metric and then the main
result to this paper is given as (Theorem \ref{teo 3.2}): A curve
$\alpha:I\subset \mathbb{R\rightarrow}G$ with the Frenet apparatus $\left \{
T,N,B,\varkappa,\tau \right \}  $ is a Bertrand curve if and only if%
\[
\lambda \varkappa+\mu \varkappa H=1
\]
where $\lambda$, $\mu$ are constants and $H$ is the harmonic curvature
function of the curve $\alpha.$


Note that three dimensional Lie groups admitting bi-invariant metrics are
$SO\left(  3\right)  ,SU^{2}$ and Abellian Lie groups. So we believe that our
characterizations about Bertrand curves will be useful for curves theory in
Lie groups.

\section{Preliminaries}

Let $G$ be a Lie group with a bi-invariant metric $\left \langle \text{
},\right \rangle $ and $D$ be the Levi-Civita connection of Lie group $G.$ If
$\mathfrak{g}$ denotes the Lie algebra of $G$ then we know that $\mathfrak{g}
$ is issomorphic to $T_{e}G$ where $e$ is neutral element of $G.$ If
$\left \langle \text{ },\right \rangle $ is a bi-invariant metric on $G$ then we
have%
\begin{equation}
\left \langle X,\left[  Y,Z\right]  \right \rangle =\left \langle \left[
X,Y\right]  ,Z\right \rangle \label{2-1}%
\end{equation}
and
\begin{equation}
D_{X}Y=\frac{1}{2}\left[  X,Y\right]  \label{2-2}%
\end{equation}
for all $X,Y$ and $Z\in \mathfrak{g}.$


Let $\alpha:I\subset \mathbb{R\rightarrow}G$ be an arc-lenghted regular curve
and $\left \{  X_{1},X_{2,}...,X_{n}\right \}  $ be an orthonormal basis of
$\mathfrak{g}.$ In this case, we write that any two vector fields $W$ and $Z$
along the curve $\alpha \ $as $W=\sum_{i=1}^{n}w_{i}X_{i}$ and $Z=\sum
_{i=1}^{n}z_{i}X_{i}$ where $w_{i}:I\rightarrow \mathbb{R}$ and $z_{i}%
:I\rightarrow \mathbb{R}$ are smooth functions. Also the Lie bracket of two
vector fields $W$ and $Z$ is given
\[
\left[  W,Z\right]  =\sum_{i=1}^{n}w_{i}z_{i}\left[  X_{i},X_{j}\right]
\]
and the covariant derivative of $W$ along the curve $\alpha$ with the notation
$D_{\alpha^{\shortmid}}W$ is given as follows%
\begin{equation}
D_{\alpha^{\shortmid}}W=\overset{\cdot}{W}+\frac{1}{2}\left[  T,W\right]
\label{2-3}%
\end{equation}
where $T=\alpha^{\prime}$ and $\overset{\cdot}{W}=\sum_{i=1}^{n}\overset
{\cdot}{w_{i}}X_{i}$ or $\overset{\cdot}{W}=\sum_{i=1}^{n}\frac{dw}{dt}X_{i}.$
Note that if $W$ is the left-invariant vector field to the curve $\alpha$ then
$\overset{\cdot}{W}=0$ (see for details \cite{crouch}).


Let $G$ be a three dimensional Lie group and $\left(  T,N,B,\varkappa
,\tau \right)  $ denote the Frenet apparatus of the curve $\alpha$. Then the
Serret-Frenet formulas of the curve $\alpha$ satisfies:%

\[
D_{T}T=\varkappa N\text{, \  \  \ }D_{T}N=-\varkappa T+\tau B\text{,
\  \  \ }D_{T}B=-\tau N
\]
where $D$ is Levi-Civita connection of Lie group $G$ and $\varkappa
=\overset{\cdot}{\left \Vert T\right \Vert }.$


\begin{definition}
\label{tan 2.1}Let $\alpha:I\subset \mathbb{R\rightarrow}G$ be a parametrized
curve. Then $\alpha$ is called a \emph{general helix} if it makes a constant
angle with a left-invariant vector field $X$. That is,%
\[
\left \langle T(s),X\right \rangle =\cos \theta \text{ for all }s\in I,
\]
for the left-invariant vector field $X\in g$ is unit length and $\theta$ is a
constant angle between $X$ and $T$, which is the tangent vector field of the
curve $\alpha$ (see \cite{ciftci}).
\end{definition}


\begin{definition}
\label{tan 2.2}Let $\alpha:I\subset \mathbb{R\rightarrow}G$ be a parametrized
curve with the Frenet apparatus $\left(  T,N,B,\varkappa,\tau \right)  $ then
\begin{equation}
\tau_{G}=\frac{1}{2}\left \langle \left[  T,N\right]  ,B\right \rangle
\label{2-4}%
\end{equation}
or
\[
\tau_{G}=\frac{1}{2\varkappa^{2}\tau}\overset{\cdot \cdot \text{
\  \  \  \  \  \  \  \ }\cdot}{\left \langle T,\left[  T,T\right]  \right \rangle
}+\frac{1}{4\varkappa^{2}\tau}\overset{\text{ \  \ }\cdot}{\left \Vert \left[
T,T\right]  \right \Vert ^{2}}%
\]
(see \cite{ciftci}).
\end{definition}


\begin{definition}
\label{tan 2.3}Let $\alpha:I\subset \mathbb{R\rightarrow}G$ be an arc length
parametrized curve. Then $\alpha$ is called a \emph{slant helix} if its
principal normal vector field makes a constant angle with a left-invariant
vector field $X$ which is unit length. That is,%
\[
\left \langle N(s),X\right \rangle =\cos \theta \text{ for all }s\in I,
\]
where $\theta \neq \frac{\pi}{2}$ is a constant angle between $X$ and $N$ which
is the principal normal vector field of the curve $\alpha$ (see \cite{zeki}).
\end{definition}

\begin{definition}
\label{tan 2.4}Let $\alpha:I\subset \mathbb{R\rightarrow}G$ be an arc length
parametrized curve with the Frenet apparatus $\left \{  T,N,B,\varkappa
,\tau \right \}  .$ Then the \emph{harmonic curvature function} of the curve
$\alpha$ is defined by%
\[
H=\dfrac{\tau-\tau_{G}}{\varkappa}%
\]
where $\tau_{G}=\frac{1}{2}\left \langle \left[  T,N\right]  ,B\right \rangle $
(see \cite{zeki})$.$

\begin{theorem}
\label{teo 2.1}Let $\alpha:I\subset \mathbb{R\rightarrow}G$ be a parametrized
curve with the Frenet apparatus $\left(  T,N,B,\varkappa,\tau \right)  $. If
the curve $\alpha$ is a general helix, if and only if%
\[
\tau=c\varkappa+\tau_{G}%
\]
where c is a constant (see \cite{ciftci}) or using the definition of the
harmonic curvature function of the curve $\alpha$ (see \cite{zeki}) is
constant function.
\end{theorem}

\begin{theorem}
Let $\alpha:I\subset \mathbb{R\rightarrow}G$ be a parametrized curve with the
Frenet apparatus $\left(  T,N,B,\varkappa,\tau \right)  $. If the curve
$\alpha$ is a general helix, if and only if the harmonic curvature function of
the curve $\alpha$ is a constant function.
\end{theorem}

\begin{proof}
It is obvious using the Definition \ref{tan 2.4} and Theorem \ref{tan 2.1}.
\end{proof}
\end{definition}

\begin{theorem}
\label{teo 2.2}Let $\alpha:I\subset \mathbb{R\rightarrow}G$ \ be a unit speed
curve with the Frenet apparatus $\left(  T,N,B,\varkappa,\tau \right)  $. Then
$\alpha$ is a slant helix if and only if%
\[
\sigma_{N}=\frac{\varkappa(1+H^{2})^{\frac{3}{2}}}{H^{\shortmid}}=\tan \theta
\]
is a constant where $H$ is a harmonic curvature function of the curve $\alpha$
and $\theta \neq \frac{\pi}{2}$ is a constant (see \cite{zeki}).
\end{theorem}

\begin{theorem}
Let $\alpha:I\subset \mathbb{R\rightarrow}G$ \ be a parametrized curve with arc
length parameter $s$ and the Frenet apparatus $\left(  T,N,B,\varkappa
,\tau \right)  $. Then, $\alpha$ is Mannheim curve if and only if
\begin{equation}
\lambda \varkappa \left(  1+H^{2}\right)  =1,\text{ for all }s\in I
\end{equation}
where $\lambda$ is constant and $H$ is the harmonic curvature function of the
curve $\alpha$ (see \cite{ismail})$.$
\end{theorem}

\begin{theorem}
Let $\alpha:I\subset \mathbb{R\rightarrow}G$ \ be a parametrized curve with arc
length parameter $s$. Then $\beta$ is the Mannheim partner curve of $\alpha$
if and only if the curvature $\varkappa_{\beta}$ and the torsion $\tau_{\beta
}$ of $\beta$ satisfy the following equation
\[
\frac{d\varkappa_{\beta}H_{\beta}}{d\overline{s}}=\frac{\varkappa_{\beta}}%
{\mu}(1+\mu^{2}\varkappa_{\beta}^{2}H_{\beta}^{2})
\]
where $\mu$ is constant and $H_{\beta}$ is the harmonic curvature function of
the curve $\beta.$
\end{theorem}

\section{Bertrand curves in a three dimensional Lie group}

In this section, we define Bertrand curves and their characterizations are
given in a three dimensional Lie group $G$ with a bi-invariant metric
$\left \langle \text{ },\right \rangle $. Also we give some characterizations of
Bertrand curves using the special cases of $G$.


\begin{definition}
\label{tan 3.1}A curve $\alpha$ in $3$-dimensional Lie group $G$ is a
\emph{Bertrand curve }if there exists a special curve\emph{\ }$\beta$ in
$3$-dimensional Lie group $G$ such that principal normal vector field of
$\alpha$ is linearly dependent principal normal vector field of $\beta$ at
corresponding point under $\psi$ which is bijection from $\alpha$ to $\beta.$
In this case $\beta$ is called the $\emph{Bertrand}$ $\emph{mate}$
$\emph{curve}$ of $\alpha$ and $\left(  \alpha,\beta \right)  $ is called
$\emph{Bertrand}$ $\emph{curve}$\emph{\ couple.}
\end{definition}

The curve $\alpha:I\subset \mathbb{R\rightarrow}G$ in $3$-dimensional Lie group
$G$ is parametrized by the arc-length parameter $s$ and from the Definition
\ref{tan 3.1} Bertrand mate curve of $\alpha$ is given $\beta:\overline
{I}\subset \mathbb{R\rightarrow}G$ in $3$-dimensional Lie group $G$ with the
help of Figure 1 such that%
\[
\text{\texttt{Figure1}}\mathtt{:}\text{{}Bertrand curve\ couple }\left(
\alpha,\beta \right)
\]%
\[
\beta \left(  s\right)  =\alpha \left(  s\right)  +\lambda \left(  s\right)
N\left(  s\right)  ,\text{ }s\in I
\]
where $\lambda$ is a smooth function on $I$ and $N$ is the principal normal
vector field of $\alpha$. We should remark that the parameter $s$ generally is
not an arc-length parameter of $\beta.$ So, we define the arc-length parameter
of the curve $\beta$ by
\[
\overline{s}=\psi \left(  s\right)  =\int \limits_{0}^{s}\left \Vert \frac
{d\beta \left(  s\right)  }{ds}\right \Vert ds
\]
where $\psi:I\longrightarrow \overline{I}$ is a smooth function and holds the
following equality%
\begin{equation}
\psi^{\prime}\left(  s\right)  =\varkappa H\sqrt{\lambda^{2}+\mu^{2}%
}\label{3-1}%
\end{equation}
for $s\in I.$

\begin{proposition}
\label{prop 3.1}Let $\alpha:I\subset \mathbb{R\rightarrow}G$ be an arc length
parametrized curve with the Frenet apparatus $\left \{  T,N,B\right \}  $. Then
the following equalities%
\begin{align*}
\left[  T,N\right]   &  =\left \langle \left[  T,N\right]  ,B\right \rangle
B=2\tau_{G}B\\
\left[  T,B\right]   &  =\left \langle \left[  T,B\right]  ,N\right \rangle
N=-2\tau_{G}N
\end{align*}
hold \cite{zeki}.
\end{proposition}

\begin{theorem}
\label{teo 3.1}Let $\alpha:I\subset \mathbb{R\rightarrow}G$ and $\beta
:\overline{I}\subset \mathbb{R\rightarrow}G$ be a Bertrand curve couple with
arc-length parameter $s$ and $\overline{s},$ respectively. Then corresponding
points are a fixed distance apart for all $s\in I$, that is,
\[
d\left(  \alpha \left(  s\right)  ,\beta \left(  s\right)  \right)
=\text{constant, \  \ for all }s\in I
\]

\end{theorem}

\begin{proof}
From Definition \ref{tan 3.1}, we can simply write
\begin{equation}
\beta \left(  s\right)  =\alpha \left(  s\right)  +\lambda \left(  s\right)
N\left(  s\right)  \label{3-2}%
\end{equation}
Differentiating the Eq. \eqref{3-2} with respect to $s$ and using the Eq.
\eqref{2-3}, we get%
\begin{align*}
\frac{d\beta \left(  \overline{s}\right)  }{d\overline{s}}\psi^{\prime}\left(
s\right)   &  =\frac{d\alpha \left(  s\right)  }{ds}+\lambda^{\prime}\left(
s\right)  N\left(  s\right)  +\lambda \left(  s\right)  \overset{\cdot}{N}\\
&  =\left(  1-\lambda \left(  s\right)  \varkappa \left(  s\right)  \right)
T(s)+\lambda^{\prime}\left(  s\right)  N\left(  s\right)  +\lambda \left(
s\right)  \tau \left(  s\right)  B\left(  s\right)  -\dfrac{1}{2}\left[
T,N\right]
\end{align*}
and with the help of Proposition \ref{prop 3.1}, we obtain
\[
\frac{d\beta \left(  \overline{s}\right)  }{d\overline{s}}\psi^{\prime}\left(
s\right)  =\left(  1-\lambda \left(  s\right)  \varkappa \left(  s\right)
\right)  T(s)+\lambda^{\prime}\left(  s\right)  N\left(  s\right)
+\lambda \left(  s\right)  \left(  \left(  \tau-\tau_{G}\right)  \left(
s\right)  \right)  B\left(  s\right)
\]
or%
\[
T_{\beta}\left(  \overline{s}\right)  =\frac{1}{\psi^{\prime}\left(  s\right)
}\left[  \left(  1-\lambda \left(  s\right)  \varkappa \left(  s\right)
\right)  T(s)+\lambda^{\prime}\left(  s\right)  N\left(  s\right)
+\lambda \left(  s\right)  \left(  \left(  \tau-\tau_{G}\right)  \left(
s\right)  \right)  B\left(  s\right)  \right]  .
\]
And then, we know that $\left \{  N_{\beta}(\left(  \overline{s}\right)
),N\left(  s\right)  \right \}  $ is a linearly dependent set, so we have
\[
\left \langle T_{\beta}\left(  \overline{s}\right)  ,N_{\beta}\left(
\overline{s}\right)  \right \rangle =\frac{1}{\psi^{\prime}\left(  s\right)
}\left[
\begin{array}
[c]{c}%
\left(  1-\lambda \left(  s\right)  \varkappa \left(  s\right)  \right)
\left \langle T(s),N_{\beta}\left(  \overline{s}\right)  \right \rangle
+\lambda^{\prime}\left(  s\right)  \left \langle N(s),N_{\beta}\left(
\overline{s}\right)  \right \rangle \\
+\lambda \left(  s\right)  \tau \left(  s\right)  \left \langle B(s),N_{\beta
}\left(  \overline{s}\right)  \right \rangle
\end{array}
\right]
\]%
\[
\lambda^{\prime}\left(  s\right)  =0
\]
that is, $\lambda \left(  s\right)  $ is constant function on $I.$ This
completes the proof.
\end{proof}

\begin{theorem}
\label{teo 3.2}If $\alpha:I\subset \mathbb{R\rightarrow}G$ \ is a parametrized
Bertrand curve with arc length parameter $s$ and the Frenet apparatus $\left(
T,N,B,\varkappa,\tau \right)  $. Then, $\alpha$ satisfy the following equality%
\begin{equation}
\lambda \varkappa \left(  s\right)  +\mu \varkappa \left(  s\right)  H\left(
s\right)  =1,\text{ for all }s\in I \label{3-3}%
\end{equation}
where $\lambda$, $\mu$ are constants and $H$ is the harmonic curvature
function of the curve $\alpha.$
\end{theorem}

\begin{proof}
Let $\alpha:I\subset \mathbb{R\rightarrow}G$ be a parametrized Bertrand curve
with arc length parameter $s$ then we can write
\[
\beta \left(  s\right)  =\alpha \left(  s\right)  +\lambda N\left(  s\right)
\]
Differentiating the above equality with respect to $s$ and by using the Frenet
equations, we get%
\begin{align*}
\frac{d\beta \left(  \overline{s}\right)  }{d\overline{s}}\psi^{\prime}\left(
s\right)   &  =\frac{d\alpha \left(  s\right)  }{ds}+\lambda \left(  s\right)
\overset{\cdot}{N}\\
&  =\left(  1-\lambda \left(  s\right)  \varkappa \left(  s\right)  \right)
T(s)+\lambda \left(  s\right)  \tau \left(  s\right)  B\left(  s\right)
-\dfrac{1}{2}\left[  T,N\right]
\end{align*}
and with the help of Proposition \ref{prop 3.1}, we obtain
\[
T_{\beta}\left(  \overline{s}\right)  =\frac{\left(  1-\lambda \varkappa \left(
s\right)  \right)  }{\psi^{\prime}\left(  s\right)  }T(s)+\frac{\lambda \left(
\left(  \tau-\tau_{G}\right)  \left(  s\right)  \right)  }{\psi^{\prime
}\left(  s\right)  }B\left(  s\right)  .
\]
As $\left \{  N_{\beta}(\left(  \overline{s}\right)  ),N\left(  s\right)
\right \}  $ is a linearly dependent set, we can write
\begin{equation}
T_{\beta}\left(  \overline{s}\right)  =\cos \theta \left(  s\right)
T(s)+\sin \theta \left(  s\right)  B(s) \label{3-4}%
\end{equation}
where
\[
\cos \theta \left(  s\right)  =\frac{\left(  1-\lambda \varkappa \left(  s\right)
\right)  }{\psi^{\prime}\left(  s\right)  },
\]%
\[
\sin \theta \left(  s\right)  =\frac{\lambda \left(  \left(  \tau-\tau
_{G}\right)  \left(  s\right)  \right)  }{\psi^{\prime}\left(  s\right)  }.
\]
If we differentiate the Eq. \eqref{3-4} and consider $\left \{  N_{\beta
}\left(  \overline{s}\right)  ,N\left(  s\right)  \right \}  $ is a linearly
dependent set we can easily see that $\theta$ is a constant function. So, we
obtain
\[
\frac{\cos \theta}{\sin \theta}=\frac{1-\lambda \varkappa \left(  s\right)
}{\lambda \left(  \left(  \tau-\tau_{G}\right)  \left(  s\right)  \right)  }%
\]
or taking $c=\dfrac{\cos \theta}{\sin \theta},$ we get
\[
\lambda \varkappa \left(  s\right)  +c\lambda \left(  \left(  \tau-\tau
_{G}\right)  \left(  s\right)  \right)  =1.
\]
Then denoting $\mu=c\lambda=$costant and using the Definition \ref{tan 2.4},
we have%
\[
\lambda \varkappa \left(  s\right)  +\mu \varkappa \left(  s\right)  H\left(
s\right)  =1,\text{ for all }s\in I
\]
which completes the proof.
\end{proof}

\begin{corollary}
The measure of the angle between the tangent vector fields of the Bertrand
curve couple $\left(  \alpha,\beta \right)  $ is constant.
\end{corollary}

\begin{proof}
It is obvious from the proof of above Theorem.
\end{proof}

\begin{remark}
It is unknown whether the reverse of the above Theorem. Because, for the proof
of the reverse we must consider a special Frenet curve $\beta \left(  s\right)
=\alpha \left(  s\right)  +\lambda N\left(  s\right)  $ in its proof. So, we
give the following Theorem.
\end{remark}

\begin{theorem}
\label{teo 3.3}Let $\alpha:I\subset \mathbb{R\rightarrow}G$ \ be a parametrized
Bertrand curve whose curvature functions $\varkappa$ and harmonic curvature
function $H$ of the curve $\alpha$ satisfy $\lambda \varkappa \left(  s\right)
+\mu \varkappa \left(  s\right)  H\left(  s\right)  =1, $ for all $s\in I$. If
the curve $\beta$ given by $\beta \left(  s\right)  =\alpha \left(  s\right)
+\lambda N\left(  s\right)  $ for all $s\in I$ is a special Frenet curve, then
$\left(  \alpha,\beta \right)  $ is the Bertrand curve couple.
\end{theorem}

\begin{proof}
Let $\alpha:I\subset \mathbb{R\rightarrow}G$ \ be a parametrized Bertrand curve
whose curvature function $\varkappa$ and harmonic curvature function $H$ of
the curve $\alpha$ satisfy $\lambda \varkappa \left(  s\right)  +\mu
\varkappa \left(  s\right)  H\left(  s\right)  =1$ for all $s\in I$. If the
curve $\beta$ given by $\beta \left(  s\right)  =\alpha \left(  s\right)
+\lambda N\left(  s\right)  $ for all $s\in I$ is a special Frenet curve, then
differentiating this equality with respect to $s$ and by using the Eq.
\eqref{3-1} with the equation $\lambda \varkappa \left(  s\right)  +\mu
\varkappa \left(  s\right)  H\left(  s\right)  =1$, we have%
\begin{equation}
T_{\beta}\left(  \overline{s}\right)  =\frac{\mu}{\sqrt{\lambda^{2}+\mu^{2}}%
}T(s)+\frac{\lambda}{\sqrt{\lambda^{2}+\mu^{2}}}B\left(  s\right)  .
\label{3-5}%
\end{equation}
Then, if we differentiate the last equation with respect to $s$ and by using
the Frenet formulas we obtain%
\begin{equation}
\varkappa_{\beta}\left(  \overline{s}\right)  N_{\beta}\left(  \overline
{s}\right)  \psi^{\prime}\left(  s\right)  =\frac{\varkappa \left(  s\right)
}{\sqrt{\lambda^{2}+\mu^{2}}}\left(  \mu-\lambda H\left(  s\right)  \right)
N\left(  s\right)  . \label{3-6}%
\end{equation}
Thus, for each $s\in I,$ the vector field $N_{\beta}\left(  \overline
{s}\right)  $ of $\beta$ is linearly dependent the vector field $N\left(
s\right)  $ of $\alpha$ at corresponding point under the bijection from
$\alpha$ to $\beta.$ This completes the proof.
\end{proof}

\begin{proposition}
\label{prop 3.2}Let $\alpha:I\subset \mathbb{R\rightarrow}G$\ be an
arc-lenghted Bertrand curve with the Frenet vector fields $\left \{
T,N,B\right \}  $ and $\beta:\overline{I}\subset \mathbb{R\rightarrow}G$ be a
Bertrand mate of $\alpha$ with the Frenet vector fields $\left \{  T_{\beta
},N_{\beta},B_{\beta}\right \}  .$ Then $\tau_{G_{\beta}}=\tau_{G}$ for the
curves $\alpha$ and $\beta$ where $\tau_{G}=\frac{1}{2}\left \langle \left[
T,N\right]  ,B\right \rangle $ and $\tau_{G\beta}=\frac{1}{2}\left \langle
\left[  T_{\beta},N_{\beta}\right]  ,B_{\beta}\right \rangle .$
\end{proposition}

\begin{proof}
Let $\alpha:I\subset \mathbb{R\rightarrow}G$\ be an arc-lenghted Bertrand curve
with the Frenet vector fields $\left \{  T,N,B\right \}  $ and $\beta
:\overline{I}\subset \mathbb{R\rightarrow}G$ be a Bertrand mate of $\alpha$
with with the Frenet vector fields $\left \{  T_{\beta},N_{\beta},B_{\beta
}\right \}  .$ From the Eq. \eqref{3-5} and considering $N_{\beta}=\mp N$ we
have%
\begin{equation}
B_{\beta}\left(  \overline{s}\right)  =-\frac{\lambda}{\sqrt{\lambda^{2}%
+\mu^{2}}}T(s)+\frac{\mu}{\sqrt{\lambda^{2}+\mu^{2}}}B\left(  s\right)  .
\label{3-7}%
\end{equation}
Since $\tau_{G\beta}=\frac{1}{2}\left \langle \left[  T_{\beta},N_{\beta
}\right]  ,B_{\beta}\right \rangle $, using the equalities of the Frenet vector
fields $T_{\beta},N_{\beta}$ and $B_{\beta}$ we obtain $\tau_{G\beta}=\tau
_{G}.$ Which completes the proof.
\end{proof}

\begin{theorem}
\label{teo 3.4}Let $\alpha:I\subset \mathbb{R\rightarrow}G$ \ be a parametrized
Bertrand curve with curvature functions $\varkappa$, $\tau$ and $\beta
:\overline{I}\subset \mathbb{R\rightarrow}G$ be a Bertrand mate of $\alpha$
with curvatures functions $\varkappa_{\beta}$, $\tau_{\beta}.$ Then the
relations between these curvature functions are%
\begin{align}
\varkappa_{\beta}\left(  \overline{s}\right)   &  =\frac{\mu \varkappa \left(
s\right)  -\lambda \varkappa \left(  s\right)  H\left(  s\right)  }{\left(
\lambda^{2}+\mu^{2}\right)  H\left(  s\right)  },\label{3-8}\\
\tau_{\beta}\left(  \overline{s}\right)   &  =\frac{\lambda \varkappa \left(
s\right)  +\mu \varkappa \left(  s\right)  H\left(  s\right)  }{\left(
\lambda^{2}+\mu^{2}\right)  H\left(  s\right)  }+\tau_{G} \label{3-9}%
\end{align}

\end{theorem}

\begin{proof}
If we take the norm of the Eq. \eqref{3-6} and use the Eq. \eqref{3-1}, we get
the Eq. \eqref{3-8}. Then differentiating the Eq. \eqref{3-7}\ and using the
Frenet formulas, we have%
\begin{align*}
\overset{\cdot}{B_{\beta}}\left(  \overline{s}\right)  \psi^{\prime}\left(
s\right)   &  =-\frac{\lambda}{\sqrt{\lambda^{2}+\mu^{2}}}\overset{\cdot}%
{T}(s)+\frac{\mu}{\sqrt{\lambda^{2}+\mu^{2}}}\overset{\cdot}{B}\left(
s\right)  ,\\
&  =-\frac{\lambda}{\sqrt{\lambda^{2}+\mu^{2}}}\varkappa(s)N(s)+\frac{\mu
}{\sqrt{\lambda^{2}+\mu^{2}}}\left(  -\tau(s)N(s)-\frac{1}{2}\left[
T,B\right]  \right)
\end{align*}
In the above equality, using the Eq. \eqref{3-1} and the Proposition
\ref{prop 3.1}, we get%
\[
\left(  \tau_{\beta}-\tau_{G\beta}\right)  N_{\beta}\left(  \overline
{s}\right)  =\frac{1}{\varkappa H\left(  \lambda^{2}+\mu^{2}\right)  }\left(
\lambda \varkappa+\mu \varkappa H\right)  N(s).
\]
If we take the norm of the last equation and use the Proposition
\ref{prop 3.2}, we get the Eq. \eqref{3-9}. Which completes the proof.
\end{proof}

\begin{theorem}
\label{teo 3.5}Let $\alpha:I\subset \mathbb{R\rightarrow}G$ \ be a parametrized
curve with Frenet apparatus $\left \{  T,N,B,\varkappa,\tau \right \}  $ and
$\beta:\overline{I}\subset \mathbb{R\rightarrow}G$ be a curve with Frenet
apparatus $\left \{  T_{\beta},N_{\beta},B_{\beta},\varkappa_{\beta}%
,\tau_{\beta}\right \}  .$ If $\left(  \alpha,\beta \right)  $ is a Bertrand
curve couple then $\varkappa \varkappa_{\beta}HH_{\beta}$ is a constant function.
\end{theorem}

\begin{proof}
We assume that $\left(  \alpha,\beta \right)  $ is a Bertrand curve couple.
Then we can write%
\begin{equation}
\alpha \left(  s\right)  =\beta \left(  s\right)  -\lambda \left(  s\right)
N_{\beta}\left(  \overline{s}\right)  . \label{3-10}%
\end{equation}
If we use the similar method in the proof of Theorem \ref{teo 3.2} and
consider the Eq. \eqref{3-10}, then we can easily see that $\varkappa
\varkappa_{\beta}HH_{\beta}$ is a constant function.
\end{proof}

\begin{theorem}
\label{teo 3.6}Let $\alpha:I\subset \mathbb{R\rightarrow}G$ \ be a parametrized
Bertrand curve with Frenet apparatus $\left \{  T,N,B,\varkappa,\tau \right \}  $
and $\beta:\overline{I}\subset \mathbb{R\rightarrow}G$ be a Bertrand mate of
the curve $\alpha$ with Frenet apparatus $\left \{  T_{\beta},N_{\beta
},B_{\beta},\varkappa_{\beta},\tau_{\beta}\right \}  . $ Then $\alpha$ is a
slant helix if and only if $\beta$ is a slant helix.
\end{theorem}

\begin{proof}
Let $\sigma_{N}$ and $\sigma_{N\beta}$ be the geodesic curvatures of the
principal normal curves of $\alpha$ and $\beta,$ respectively. Then using the
Theorem \ref{teo 3.4} we can easily see that
\[
\sigma_{N\beta}=-\frac{\varkappa(1+H^{2})^{\frac{3}{2}}}{H^{\shortmid}%
}=-\sigma_{N}.
\]
So, with the help of the Theorem \ref{teo 2.2} we complete the proof.
\end{proof}

\begin{theorem}
\label{teo 3.7}Let $\alpha:I\subset \mathbb{R\rightarrow}G$ \ be a parametrized
Bertrand curve \ with curvature functios $\varkappa$, $\tau$ and
$\beta:\overline{I}\subset \mathbb{R\rightarrow}G$ be a Bertrand mate of the
curve $\alpha$ with curvature functions $\varkappa_{\beta},\tau_{\beta}.$ Then
$\alpha$ is a general helix if and only if $\beta$ is a general helix.
\end{theorem}

\begin{proof}
Let $\alpha$ be a helix. From Theorem \ref{teo 2.1}, we have $H$ is a constant
function. Then using Theorem \ref{teo 3.4}, we get%
\begin{equation}
\frac{\tau_{\beta}-\tau_{G\beta}}{\varkappa_{\beta}}=\frac{\lambda+\mu H}%
{\mu-\lambda H}. \label{3-11}%
\end{equation}
Since $H$ is constant function, the Eq. \eqref{3-11} is constant. So, $\beta$
is a general helix.

Conversly, assume that $\beta$ be a general helix. So, $\frac{\tau_{\beta
}-\tau_{G\beta}}{\varkappa_{\beta}}=$constant. From the Eq. \eqref{3-11}
$c=\frac{\lambda+\mu H}{\mu-\lambda H}=$constant and then $H=\frac
{c\mu-\lambda}{\mu+\lambda c}=$constant. Consequently $\alpha$ is a general
helix. Which completes the proof.
\end{proof}


\begin{thebibliography}{99}                                                                                               %


\bibitem {bal}H. Balgetir, M. Bekta\c{s} and J.\ Inoguchi, Null Bertrand
curves in Minkowski 3-space and their characterizations. Note Mat. 23
(2004/05), no. 1, 7-13.

\bibitem {bal1}H. Balgetir, M. Bekta\c{s} and M. Erg\"{u}t, Bertrand curves
for non null curves in 3-dimensional Lorentzian space. Hadronic J. 27 (2004),
no. 2, 229-236.

\bibitem {bert}J. M. Bertrand, M\'{e}moire sur la th\'{e}orie des courbes
\'{a} double courbure, Comptes Rendus, vol.36, 1850.

\bibitem {bioche}Ch. Bioche, Sur les courbes de M. Bertrand, Bull. Soc. Math.
France 17 (1889), 109--112.

\bibitem {burke}J. F. Burke, Bertrand Curves Associated with a Pair of Curves
Mathematics Magazine, Vol. 34, No. 1 (Sep.- Oct.,1960), pp. 60-62.

\bibitem {camci}\c{C}. Camc\i, K. \.{I}larslan, L. Kula, H.H. Hac\i
saliho\u{g}lu, Harmonic curvatures and generalized helices in $E^{n}$, Chaos
Solitons Fract. 40 (2007) 1--7.

\bibitem {crouch}P. Crouch, F. Silva Leite, The dynamic interpolation problem:
on Riemannian manifoldsi Lie groups and symmetric spaces, J. Dyn. Control
Syst. 1 (2) (1995) 177-202.

\bibitem {ciftci}\"{U}. \c{C}ift\c{c}i, A generalization of Lancert's theorem,
J. Geom. Phys. 59 (2009) 1597-1603.

\bibitem {coken}A. C. \c{C}\"{o}ken, \"{U}. \c{C}ift\c{c}i, A note on the
geometry of Lie groups, Nonlinear Analysis TMA 68 (2008) 2013-2016.

\bibitem {il}N. Ekmekci and K. \.{I}larslan, On Bertrand curves and their
characterization. Differ. Geom. Dyn. Syst. 3 (2001), no. 2, 17-24.

\bibitem {gok}\.{I}. G\"{o}k, \c{C}. Camc\i, H.H. Hac\i saliho\u{g}lu,
Vn-slant helices in Euclidean n-space $E^{n}$, Math. Commun. 14 (2) (2009) 317--329.

\bibitem {ismail}\.{I}. G\"{o}k, O. Z. Okuyucu, N. Ekmekci\ and Y. Yayl\i, On
Mannheim partner curves in three dimensional Lie groups, submitted, 2012.

\bibitem {izu}S. Izumiya and N. Tkeuchi, New special curves and developable
surfaces, Turk. J. Math 28 (2004), 153-163.

\bibitem {jin}D. H. Jin, Null Bertrand curves in a Lorentz manifold. J. Korea
Soc. Math. Educ. Ser. B Pure Appl. Math. 15 (2008), no. 3, 209--215.

\bibitem {kuh}W. Kuhnel, Differential geometry: curves-surfaces-manifolds,
Braunschweig, Wiesbaden, 1999.

\bibitem {kul}M. K\"{u}lahc\i \ and M. Erg\"{u}t, Bertrand curves of AW(k)-type
in Lorentzian space, Nonlinear Analysis: Theory, Methods \& Applications,
Volume 70, Issue 4, 15 February 2009, Pages 1725-1731.

\bibitem {kulahci}M. K\"{u}lahc\i,\ M. Bekta\c{s}, M. Erg\"{u}t, On Harmonic
curvatures of Frenet curve in Lorentzian space., Chaos Solitons Fract. 41
(2009) 1668-1675.

\bibitem {mat}H. Matsuda and S. Yorozu, Notes on Bertrand curves. Yokohama
Math. J. 50 (2003), no. 1-2, 41-58.

\bibitem {zeki}O. Z. Okuyucu, \.{I}. G\"{o}k, Y. Yayl\i \ and \ N. Ekmekci,
Slant Helices in three Dimensional Lie Groups, preprint 2012:
arXiv:1203.1146v2 [math.DG].

\bibitem {ozdamar}E. \"{O}zdamar, H.H. Hac\i saliho\u{g}lu, A characterization
of inclined curves in Euclidean n-space, Commun. Fac. Sci. Univ., Ser. 24A
(1975) 15--22.

\bibitem {Pears}L. R. Pears, Bertrand curves in Riemannian space, J. London
Math. Soc. Volume s1-10, Number 2, 180-183 , July 1935.

\bibitem {james}J. K. Whittemore, Bertrand curves and helices. Duke Math. J.
6, (1940). 235--245.

\bibitem {yil}M. Yildirim Yilmaz and M. Bekta\c{s}, General properties of
Bertrand curves in Riemann--Otsuki space, Nonlinear Analysis: Theory, Methods
\& Applications, Volume 69, Issue 10, 15 November 2008, Pages 3225-3231.
\end{thebibliography}
\end{document}